\documentclass[12pt,a4paper]{article}

\usepackage{amsmath}
\usepackage{amsthm,amssymb}

\usepackage{graphicx} 

\theoremstyle{plain}
\newtheorem{theorem}{Theorem}
\newtheorem{corollary}[theorem]{Corollary}

\date{\today}

\begin{document}

\title{A Bijection on Bilateral Dyck Paths}

\author{Paul R.~G.~Mortimer and Thomas Prellberg\\
\small School of Mathematical Sciences\\[-0.8ex]
\small Queen Mary University of London\\[-0.8ex] 
\small Mile End Road, London E1 4NS, UK\\
\small\tt \{p.r.g.mortimer,t.prellberg\}@qmul.ac.uk
}

\maketitle

\begin{abstract}
It is known that both the number of Dyck paths with $2n$ steps and $k$ peaks, and the number of Dyck paths with $2n$ steps and $k$ steps at odd height follow the Narayana distribution. In this paper we present a bijection which explicitly illustrates this equinumeracy. Moreover, we extend this bijection to bilateral Dyck paths. The restriction to Dyck paths preserves the number of contacts.


\end{abstract}

\section{Introduction}

\label{introduction}

Bilateral Dyck paths are directed walks on $\mathbb{Z}^2$ starting at $(0,0)$ in the $(x,y)$-plane and ending on the line $y=0$, which have steps in the $(1,1)$ (up-step) and $(1,-1)$ (down-step) directions. Let $\cal B$ be the set of all bilateral Dyck paths. Note that in the literature these are also referred to as free Dyck paths \cite{chen2009} or Grand-Dyck paths \cite{manes2012}.
Dyck paths are paths in $\cal B$ which have no vertices with negative $y$-coordinates. Let $\cal D$ be the set of all Dyck paths.  
We define a negative Dyck path to be a bilateral Dyck path of nonzero length which has no vertices with positive $y$-coordinates.

Given $\pi \in \cal B$, we define the semilength $n(\pi)$ to be half the number of
its steps. We say that an up-step is at height $j$ if it starts at a vertex $(i-1,j-1)$ and ends at a vertex $(i,j)$; it is at odd height if and only if $j$ is odd. A down-step is at height $j$ if it starts at a vertex $(i,j)$ and ends at a vertex $(i+1,j-1)$; it is at odd height if and only if $j$ is odd.
Therefore the first and last steps in any Dyck path are at height $1$, and an up-step and its matching down-step have the same height. 
This definition is consistent with that in \cite{osborn2010}. 
We define a peak as an up-step followed immediately by a down-step, and the height of a peak as the height of the steps which form it, or equivalently as the height of the vertex common to both steps of the peak. We define a valley to be a down-step followed immediately by an up-step.
We define a contact as a down-step at height $1$ or an up-step at height $0$. We define a crossing to be either a down-step at height $1$ followed immediately by a down-step at height $0$, or an up-step at height $0$ followed immediately by an up-step at height $1$. 
We define a prime Dyck path to be a Dyck path with exactly one contact.

Let $\cal L$ be the set of words with the symbol set $\{U,D\}$. A bilateral Dyck word of semilength $n$ is an element of $\cal L$ such that the letters $U$ and $D$ each appear $n$ times. A Dyck word of semilength $n$ is a bilateral Dyck word such that no initial segment of the string has more $D$s than $U$s. There is an obvious bijection between paths and words, mapping up-steps to $U$s and down-steps to $D$s, and we will use words and paths interchangeably in this paper. For example, we say that a Dyck word is prime if it corresponds to a prime Dyck path.

It is an established fact that many patterns in Dyck paths are enumerated by Narayana numbers
$$N(n,k)=\frac1n\binom nk\binom n{k-1}\;.$$
For example, there are $N(n,k)$ Dyck paths of semilength $n$ with precisely $k$ peaks, or $k$ up-steps at odd height, or $k-1$ up-steps at even height, and many other statistics \cite{kreweras1986,sulanke1998}.
These are most easily proven by generating function techniques \cite{deutsch1999}. Equinumeracy then follows from showing that different counting problems have the same generating function.

In this note we give a direct bijection between Dyck paths of semilength $n$ with $k$ up-steps at odd height and Dyck paths of semilength $n$ with $k$ peaks. Clearly such a bijection provides more structural insight than a proof using generating function techniques. 

In \cite{osborn2010} a bijection between Dyck paths of semilength $n$ with $k-1$ up-steps at even height and Dyck paths of semilength $n$ with $k$ peaks is given, using the machinery of checkmark sequences. That bijection is given as a restriction of a bijection between bilateral Dyck paths.
 
To make a connection of our result with the work in \cite{osborn2010}, we remark
that any Dyck path with $k$ up-steps at odd height can be mapped bijectively to a Dyck
path with $k-1$ up-steps at even height as follows. If $P$ is a Dyck path of nonzero length, the
corresponding Dyck word $W_P$ can be uniquely decomposed into a word of the form $UW_1DW_2$ where
$W_1$ and $W_2$ are Dyck words. Exchanging $W_1$ and $W_2$ changes the parity of the steps of the sub-paths corresponding to $W_1$ and $W_2$, 
and therefore the involution mapping $UW_1DW_2$ to $UW_2DW_1$ gives the desired bijection.\footnote{We have not been able to find this simple
argument in the literature.} We will define this more carefully in Section \ref{extension}.

By bijecting Dyck paths of semilength $n$ with $k-1$ up-steps at even height to Dyck paths with $k$ up-steps at odd height using the involution just described, followed by application of the map $\phi$ given in the theorem below, we also have a bijection of Dyck paths of semilength $n$ with $k-1$ up-steps at even height to Dyck paths of semilength $n$ with $k$ peaks. Already from inspecting the five Dyck paths of semilength $n=3$ one can see that this bijection is different from the one presented in \cite{osborn2010}.

We note that the extension of the bijection in \cite{osborn2010} to bilateral Dyck paths needs a slightly modified definition of peaks. More precisely, in \cite{osborn2010} an initial down-step or final up-step in a bilateral Dyck path is also counted as a peak, and the bijection given is between bilateral Dyck paths of semilength $n$ with $k-1$ up-steps at even height and bilateral Dyck paths of semilength $n$ with $k$ (modified) peaks.

In contrast, our bijection extends to a bijection between bilateral Dyck paths of semilength $n$ with $k$ up-steps at odd height and bilateral Dyck paths of semilength $n$ with $k$ peaks.

Note that the restriction of our bijection to Dyck paths is essentially identical to a bijection given in \cite{sapounakis2009}. There, the authors used two length-preserving bijections on Dyck paths to show the equidistribution of statistics of certain strings occurring at odd height, at even height, and anywhere. However, the problem discussed here, regarding the equidistribution of up-steps at odd height and peaks anywhere, was not addressed in \cite{sapounakis2009}.

\section{The Bijection on Dyck Paths}

\label{result}

If $P$ is a Dyck path of nonzero length with $s \geq 0$ down steps at height $2$ before the first contact, the corresponding Dyck word $W_P$ can be uniquely decomposed into a word of the form $UUW_1DUW_2DU\ldots W_sDDW_{s+1} =  U \Big( \prod_{i=1}^s U W_i D \Big) DW_{s+1} $ where each $W_i$ is a Dyck word. 
Using this decomposition, we define the map $\phi: \cal D \rightarrow \cal D$ recursively by setting
$$\phi(W_P) = \phi \Big( U \Big( \prod_{i=1}^s U W_i D \Big) D  W_{s+1} \Big) = 
 \Big( \prod_{i=1}^s U \phi(W_i) \Big) (UD)D^s\phi(W_{s+1})\;.$$
and $\phi(\epsilon) = \epsilon$, where $\epsilon$ denotes the empty word. 

As the decomposition expresses a Dyck word in terms of strictly smaller Dyck words, the recursion terminates. Therefore $\phi$ is well-defined on all Dyck words.

\begin{theorem}
\label{theorem1}
$\phi$ gives an explicit bijection from the set of Dyck paths of semilength $n$ with $m$ contacts and $k$ up-steps at odd height to the set of Dyck paths of semilength $n$ with $m$ contacts and $k$ peaks.
\end{theorem}

\begin{proof}

The decomposition implies that the parity of a step in the path corresponding to the sub-word 
$W_i$ is the same as its parity in the original path $P$.
As the first step of each sub-path corresponding to nonempty $W_i$ is at odd height, each recursive iteration either maps the empty word to itself or
maps exactly one up-step at odd height (the first step of the path $P$) to a peak. 
Thus $\phi$ maps a path with $k$ up-steps at odd height to a path with $k$ peaks. 
Clearly, $\phi$ does not change the number of contacts.

We now show that $\phi$ is indeed a bijection by giving its inverse. If $P$ is a Dyck path with right-most peak at height $s+1$ with $s\geq0$, the corresponding Dyck word $W_P$ can be uniquely decomposed into a word of the form $UW_1UW_2U\ldots UW_{s}(UD)D^{s}W_{s+1} = \Big( \prod_{i=1}^{s} U W_i \Big) UD^{s+1}W_{s+1} $ where each $W_i$ is a Dyck word. Let $\psi: \cal D \rightarrow \cal D$ be defined recursively as
$$\psi(W_P) = \psi \Big( \Big( \prod_{i=1}^{s} U W_i \Big) UD^{s+1}W_{s+1}\Big) =  U \Big( \prod_{i=1}^{s} U \psi(W_i) D \Big) D\,\psi(W_{s+1}) $$
with  $\phi(\epsilon) = \epsilon$, where $\epsilon$ denotes the empty word.
We note in passing that each recursive iteration of $\psi$ to a nonempty path maps exactly one peak to an up-step at odd height.

To show that $\psi$ is the inverse of $\phi$, we proceed by induction on $n$, the semilength of word. 
For $n=0$, 
$$\psi(\phi(\epsilon))=\psi(\epsilon)=\epsilon\quad\mbox{and}\quad\phi(\psi(\epsilon))=\phi(\epsilon)=\epsilon\;.$$
Now take a Dyck word $W$ of semilength $n>0$ and assume that $\psi(\phi(W')) = \phi(\psi(W')) = W'$ for all words $W'$ of semilength less than $n$. $W$ can be decomposed as either $W =  U \Big( \prod_{i=1}^s U W_i D \Big) D W_{s+1} $ or $W = \Big( \prod_{i=1}^s U W_i \Big) UD^{s+1}W_{s+1} $, where the $W_i$ are Dyck words of semilength strictly less than $n$. Hence, by the inductive hypothesis,

\begin{align*} \psi(\phi(W)) & = \psi\Big(\phi \Big( U \Big( \prod_{i=1}^s U W_i D \Big) D  W_{s+1} \Big) \Big)  =  \psi\Big(\Big( \prod_{i=1}^s U \phi(W_i) \Big) UD^{s+1}   \phi(W_{s+1}) \Big) \\ 
 & =  U \Big( \prod_{i=1}^s U \psi(\phi(W_i)) D \Big) D \, \psi(\phi(W_{s+1}))  =  U \Big( \prod_{i=1}^s U W_i D \Big) D  W_{s+1}  = W_P\;\text{, and}\\
\phi(\psi(W)) & = \phi\Big(\psi \Big( \Big( \prod_{i=1}^s U W_i \Big) UD^{s+1} W_{s+1}\Big) \Big) = \phi \Big( U \Big( \prod_{i=1}^s U \psi(W_i) D \Big) D \,\psi(W_{s+1}) \Big)\\
& = \Big( \prod_{i=1}^s U \phi(\psi(W_i)) \Big) UD^{s+1} \phi(\psi(W_{s+1})) = \Big(\prod_{i=1}^s U W_i \Big) UD^{s+1} W_{s+1} = W_P\;.
\end{align*}

\end{proof}

\begin{corollary}
The number of Dyck paths of semilength $n$ with $m$ contacts and $k$ up-steps at odd height is equal to the number of Dyck paths of semilength $n$ with $m$ contacts and $k$ peaks.
\end{corollary}

\section{An Example}

\begin{figure}[ht!]
\begin{center}\includegraphics[width=0.9\textwidth]{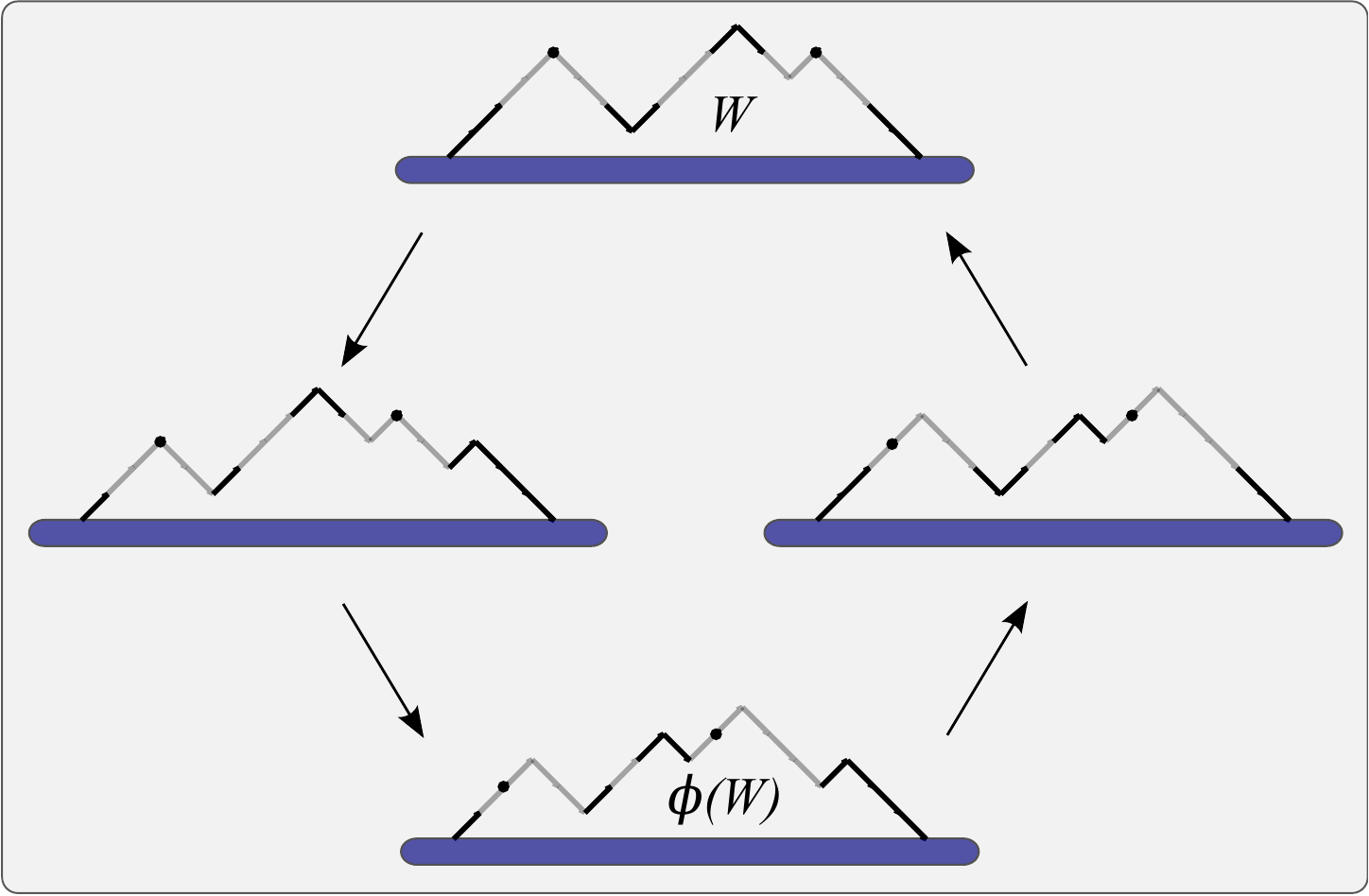}\end{center}
\caption{$\phi$ mapping a Dyck path $W$ (top) with $18$ steps, $4$ up-steps at odd height and $1$ contact to a Dyck path $\phi(W)$ (bottom) with $18$ steps, $4$ peaks and $1$ contact, and $\psi$ performing the inverse mapping. Intermediate stages after each recursive iteration of $\phi$ and $\psi$ are shown on the left and right, respectively. Black dots correspond to the occurrence of empty words in the decomposition of the corresponding words.}
\label{fig}
\end{figure}

An example of the bijection $\phi$ and its inverse $\psi$ is given in Figure \ref{fig}, explicitly showing the intermediate stages after each recursive iteration. When writing the corresponding words below, we
insert bracketing corresponding to the decompositions, e.g.
we write 
$$UUUUDDDUUUUDDUDDDD=UU\Big(UU\big(\big)DD\Big)DU\Big(UU\big(UD\big)DU\big(\big)DD\Big)DD\;.$$

Note in particular
that the double bracket $()$ signifies an empty word in the decomposition. For the sake of legibility, we have avoided the occurrence of composite words in this example.

The mapping $\phi$ corresponding to the left arrows in Figure \ref{fig} is now performed as follows.

\begin{align*}
\phi\bigg(UU\Big(UU\big(&\big)DD\Big)DU\Big(UU\big(UD\big)DU\big(\big)DD\Big)DD\bigg)&\text{top path}\\
&=U\phi\Big(UU\big(\big)DD\Big)U\phi\Big(UU\big(UD\big)DU\big(\big)DD\Big)UDDD&\text{left path}\\
&=U\Big(U\phi\big(\big)UDD\Big)U\Big(U\phi\big(UD\big)U\phi\big(\big)UDDD\Big)UDDD&\text{bottom path}\\
&=U\Big(U\big(\big)UDD\Big)U\Big(U\big(UD\big)U\big(\big)UDDD\Big)UDDD&\text{bottom path}
\end{align*}

The first recursive iteration of $\phi$ gives
$$\phi(UUW_1DUW_2DD)=U\phi(W_1)U\phi(W_2)UDDD\;,$$
and the second step applies the recursion to the sub-words $W_1$ and $W_2$ respectively. 
In the third step the recursion is applied to empty subwords, and hence the path no longer changes; the corresponding figures look identical. 
Note in particular that while $\phi(UU()DD)=U\phi()UDD=U()UDD$ does not change the word $UUDD$,
the position of the empty word $()$, and correspondingly the position of
the black dot in Figure \ref{fig} changes.

The inverse mapping $\psi$ corresponding to the right arrows in Figure \ref{fig} is performed similarly, with different intermediate stages.
 
\begin{align*}
\psi\bigg(U\Big(U\big(\big)&UDD\Big)U\Big(U\big(UD\big)U\big(\big)UDDD\Big)UDDD\bigg)&\text{bottom path}\\
&=UU\psi\Big(U\big(\big)UDD\Big)DU\psi\Big(U\big(UD\big)U\big(\big)UDDD\Big)DD&\text{right path}\\
&=UU\Big(UU\psi\big(\big)DD\Big)DU\Big(UU\psi\big(UD\big)DU\psi\big(\big)DD\Big)DD&\text{top path}\\
&=UU\Big(UU\big(\big)DD\Big)DU\Big(UU\big(UD\big)DU\big(\big)DD\Big)DD&\text{top path}
\end{align*}

\section{Extension to Bilateral Dyck Paths}

\label{extension}

We now give an extension of $\phi$ to bilateral Dyck paths. Before we do this, we must define the following two maps.

We define the map $\alpha: \cal B \rightarrow \cal B$ as follows.
If $P$ is a bilateral Dyck path of semilength $n$ with corresponding bilateral Dyck word $W_P = A_1A_2 \ldots A_{2n}$, where $A_i \in \{U,D\}$, then
(i) $\alpha(U)=D$, (ii) $\alpha(D)=U$ and (iii) $\alpha(W_P) = \alpha \Big( \prod\limits_{i=1}^{2n} A_i  \Big) = \prod\limits_{i=1}^{2n}  \alpha(A_i)$.

The map $\alpha$ reflects bilateral Dyck paths in the $x-$axis. In particular, it maps nonempty Dyck paths to negative Dyck paths and vice versa. It maps steps at odd height to steps at even height, and vice versa. It also maps peaks to valleys, and maps valleys to peaks.

We define $\beta: \cal D \rightarrow \cal D$ as follows.
If $P$ is a Dyck path of nonzero length, the corresponding Dyck word $W_P$ can be uniquely decomposed into a word of the form $UW_1DW_2$ where $W_1$ and $W_2$ are Dyck words. Then 

$$ \beta(W_P) = \beta(UW_1DW_2) = UW_2DW_1\;.$$

The map $\beta$ is the one mentioned in Section \ref{introduction}, which illustrates the equinumeracy between paths with $k+1$ up-steps at odd height and $k$ up-steps at even height; excepting the first up- and down-steps at height $1$, $\beta$ maps steps at odd height to steps at even height, and vice versa. It preserves the number of peaks or valleys, but does not preserve the number of contacts. We note that both $\alpha$ and $\beta$ are involutions.
We are now in a position to state the extended bijection.

Let $\epsilon$ denote the empty word. We define $\phi': \cal B \rightarrow \cal B$ in the following way:

\begin{itemize}
\item[(a)] If $P$ is a Dyck path with corresponding Dyck word $W_P$, then 
 
$$\phi'(W_P) = \phi(W_P)\;.$$

\item[(b)] If $P$ is a negative Dyck path with corresponding negative Dyck word $W_P$, then

$$\phi'(W_P) = \alpha(\phi(\beta(\alpha(W_P))))\;.$$

\item[(c)] If $P$ is a bilateral Dyck path with $l>0$ crossings then we can uniquely decompose the corresponding bilateral Dyck word $W_P$ into $W_1W_2\ldots W_{l+1} = \prod_{i=1}^{l+1}  W_i$, where each $W_i$ is alternately either a Dyck word of nonzero length or a negative Dyck word. Then

 $$\phi'(W_P) = \phi' \Big( \prod_{i=1}^{l+1}  W_i  \Big) = \prod_{i=1}^{l+1} \phi'(W_i)\;.$$
 
 \end{itemize}
 
\begin{theorem}
$\phi'$ gives an explicit bijection from the set of bilateral Dyck paths of semilength $n$ with $k$ up-steps at odd height to the set of bilateral Dyck paths of semilength $n$ with $k$ peaks.
\end{theorem}

\begin{proof} 
By Theorem \ref{theorem1}, (a) maps Dyck paths with $k$ up-steps at odd height to Dyck paths with $k$ peaks. 
 
If $P$ is a negative Dyck path with $k$ up-steps at odd height and corresponding negative Dyck word $W_P$, then $\alpha(W_P)$ corresponds to a Dyck path with $k$ up-steps at even height. Applying $\beta$ to this new path gives a Dyck path with $k+1$ up-steps at odd height. Then $\phi(\beta(\alpha(W_P)))$ corresponds to a Dyck path with $k+1$ peaks (and so $k$ valleys). Finally, applying $\alpha$ again will give a negative Dyck path with $k$ peaks. Thus (b) maps negative Dyck paths with $k$ up-steps at odd height to negative Dyck paths with $k$ peaks.

Application of (c) does not alter the parity of steps or the number of peaks. 
Thus $\phi'$ maps bilateral Dyck paths with $k$ up-steps at odd height to bilateral Dyck paths with $k$ peaks.

We now show that $\phi'$ is a bijection by giving its inverse. Let $\psi': \cal B \rightarrow \cal B$, be defined as follows.

\begin{itemize}

\item [(a')] If $P$ is a Dyck path with corresponding Dyck word $W_P$, then 
 
$$\psi'(W_P) = \psi(W_P)\;.$$

\item [(b')] If $P$ is a negative Dyck path with corresponding negative Dyck word $W_P$, then

$$\psi'(W_P) = \alpha(\beta(\psi(\alpha(W_P))))\;.$$

\item [(c')] If $P$ is a bilateral Dyck path with $l>0$ crossings then we can uniquely decompose the corresponding bilateral Dyck word $W_P$ into $W_1W_2\ldots W_{l+1} = \prod_{i=1}^{l+1}  W_i$, where each $W_i$ is alternately either a Dyck word of nonzero length or a negative Dyck word. Then

 $$\psi'(W_P) = \psi' \Big( \prod_{i=1}^{l+1}  W_i  \Big) = \prod_{i=1}^{l+1} \psi'(W_i)\;.$$

\end{itemize}

To show that $\psi'$ is the inverse of $\phi'$, we proceed as follows. Let $W$ be a bilateral Dyck word. If the path associated to $W$ has $l>0$ crossings then we can uniquely decompose $W$ into $W_1W_2\ldots W_{l+1} = \prod_{i=1}^{l+1}  W_i$, where each $W_i$ is alternately either a Dyck word of nonzero length or a negative Dyck word. Then

\begin{align*}
 \psi'(\phi'(W)) &= \psi'\Big(\phi' \Big( \prod_{i=1}^{l+1}  W_i  \Big)\Big) = \psi'\Big(\prod_{i=1}^{l+1} \phi'(W_i)\Big) = \prod_{i=1}^{l+1} \psi'(\phi'(W_i))\text{, and}\\
\phi'(\psi'(W)) &= \phi'\Big(\psi' \Big( \prod_{i=1}^{l+1}  W_i  \Big)\Big) = \phi'\Big(\prod_{i=1}^{l+1} \psi'(W_i)\Big) = \prod_{i=1}^{l+1} \phi'(\psi'(W_i))\;.
  \end{align*}
 
 If $W$ is a Dyck word then $\phi'=\phi$ and $\psi'=\psi$ so 
 \begin{align*}
 \psi'(\phi'(W)) &= \psi(\phi(W)) = W \text{ and}\\
 \phi'(\psi'(W)) &= \phi(\psi(W)) = W\;.
 \end{align*}

Using the fact that $\phi$ and $\psi$ are inverses of each other, together with the involutive properties of $\alpha$ and $\beta$, it is easy to show that for negative Dyck words $\psi'(\phi'(W))=W=\phi'(\psi'(W))$ also holds:

\begin{align*}
\psi'(\phi'(W)) &= \psi'(\alpha(\phi(\beta(\alpha(W))))
=\alpha(\beta(\psi(\alpha(\alpha(\phi(\beta(\alpha(W))))))))\\    &= \alpha(\beta(\psi(\phi(\beta(\alpha(W)))))) = \alpha(\beta(\beta(\alpha(W)))) = \alpha(\alpha(W)) = W \text{, and}\\
\phi'(\psi'(W)) &= \phi'(\alpha(\beta(\psi(\alpha(W))))
=\alpha(\phi(\beta(\alpha(\alpha(\beta(\psi(\alpha(W))))))))\\    &= \alpha(\phi(\beta(\beta(\psi(\alpha(W)))))) = \alpha(\phi(\psi(\alpha(W)))) = \alpha(\alpha(W)) = W\;.
\end{align*}

\end{proof}

\begin{corollary}
The number of bilateral Dyck paths of semilength $n$ with $k$ up-steps at odd height is equal to the number of bilateral Dyck paths of semilength $n$ with $k$ peaks.
\end{corollary}

\section*{Acknowledgements}

The authors wish to thank Judy-anne Osborn for helpful comments on an early version of the manuscript, and the referees for their careful comments on our work.

\end{document}